\documentclass{article}
\usepackage[a4paper, total={7in, 10in}]{geometry}
\usepackage{amsmath,amsthm,amssymb,mathrsfs,amstext, titlesec,enumitem, comment, graphicx, color, xcolor, stmaryrd,mathabx, lipsum,filecontents}
\usepackage{hyperref}
\usepackage{xpatch}
\usepackage{tikz-cd}
\hypersetup{colorlinks,linkcolor={magenta},citecolor={blue}}
\usepackage{xpatch}
\usepackage{tikz-cd}
\makeatletter
\def\Ddots{\mathinner{\mkern1mu\raise\p@
\vbox{\kern7\p@\hbox{.}}\mkern2mu
\raise4\p@\hbox{.}\mkern2mu\raise7\p@\hbox{.}\mkern1mu}}
\makeatother

\titleformat{\section}
{\normalfont\Large\bfseries}{\thesection}{1em}{}
\titleformat*{\subsection}{\Large\bfseries}
\titleformat*{\subsubsection}{\large\bfseries}
\titleformat*{\paragraph}{\large\bfseries}
\titleformat*{\subparagraph}{\large\bfseries}

\makeatother
\theoremstyle{Theorem}
\newtheorem{thm}{Theorem}[section]
\newtheorem{lem}[thm]{Lemma}

\newtheorem{cor}[thm]{Corollary}

\theoremstyle{definition}
\newtheorem{defn}[thm]{Definition}
\newtheorem{note}[thm]{Note}
\newtheorem{rem}[thm]{Remark}
\newtheorem{example}[thm]{Example}

\newcommand{\N}{\mathbb{N}}

\newcommand{\Po}{\mathbb{P}}

\newcommand{\F}{\mathcal{F}}
\newcommand{\G}{\mathcal{G}}

\date{\vspace{-5ex}}

\begin{document}

\title{{\bf Nonlinear Kernel Partition Regularity: Necessary and Sufficient  Conditions}}

\author{ 
Sayan Goswami\\  \textit{ sayan92m@gmail.com}\footnote{Ramakrishna Mission Vivekananda Educational and Research Institute, Belur Math,
Howrah, West Benagal-711202, India.}
}


\makeatother

\providecommand{\corollaryname}{Corollary}
\providecommand{\definitionname}{Definition}
\providecommand{\examplename}{Example}
\providecommand{\factname}{Fact}
\providecommand{\lemmaname}{Lemma}
\providecommand{\propositionname}{Proposition}
\providecommand{\remarkname}{Remark}
\providecommand{\theoremname}{Theorem}

\maketitle
\begin{abstract}
A matrix \( A \) is called \emph{kernel partition regular} if, for every finite coloring of the natural numbers \( \mathbb{N} \), there exists a monochromatic solution to the equation \( A\vec{X} = 0 \). In 1933, Rado characterized such matrices by showing that a matrix is kernel partition regular if and only if it satisfies the so-called \emph{column condition}. In this article, we investigate polynomial extensions of Rado’s theorem by studying systems of nonlinear equations of the form $A \vec{X} + P(z) = \vec{0},$ where $A$ is a matrix with integer entries and 
$P$ is a finite set of polynomials in one variable with no constant term.  We present several nonlinear systems of equations that are kernel partition regular, showing that the classical column condition still guarantees kernel partition regularity, even when the system is extended by adding a nonlinear polynomial term. We then establish a structural necessary condition for the partition regularity of nonlinear Rado-type systems, extending the classical column condition to a nonlinear setting. This condition generalizes Rado’s classical column condition by exploring the dependencies between the linear and higher-degree polynomial components of the system.
\end{abstract}

\section{Introduction}
Van der Waerden \ achieved one of the earliest breakthroughs in combinatorial number theory cite{vdw}, who proved that every finite coloring of the natural numbers contains arbitrarily long arithmetic progressions in at least one color class. Rado significantly extended this foundational result \cite{rado}, who provided a complete characterization of homogeneous linear systems that admit monochromatic solutions under any finite partition of \( \mathbb{N} \). In this article, we use the term \emph{partition regular} to indicate that a certain structure can be found monochromatically under any finite coloring of \( \mathbb{N} \). For instance, a collection \( \mathcal{F} \) of subsets of \( \mathbb{N} \) is said to be \emph{partition regular} if, for every finite coloring of \( \mathbb{N} \), there exists a set \( F \in \mathcal{F} \) that lies entirely within a single color class. Similarly, an equation of the form \( F(x_1, \ldots, x_n) = 0 \) over \( \mathbb{N} \) is said to be \emph{partition regular} if, for every finite coloring of \( \mathbb{N} \), the equation admits a monochromatic solution.

To state Rado's classification theorem, we first recall the definition of the \emph{column condition}.

\begin{defn}[\textbf{Column Condition}]
Let \( m, n \in \mathbb{N} \), and let \( A \) be an \( m \times n \) matrix with integer entries. Denote the columns of \( A \) by \( (C_1, \ldots, C_n) \). The matrix \( A \) is said to satisfy the \emph{column condition} if there exists a partition of the index set \( \{1, 2, \ldots, n\} \) into subsets \( I_1 \cup \cdots \cup I_v \) such that:
\begin{itemize}
    \item \( \sum_{j \in I_1} C_j = \vec{0} \), and
    \item for each \( t > 1 \), the sum \( \sum_{j \in I_t} C_j \) can be expressed as a rational linear combination of the columns \( \{ C_j : j \in I_1 \cup \cdots \cup I_{t-1} \} \).
\end{itemize}
\end{defn}

A matrix \( A \) is called \emph{kernel partition regular} if, for every finite coloring of \( \mathbb{N} \), the equation \( A\vec{X} = 0 \) admits a monochromatic solution. Rado's theorem asserts that a matrix \( A \) is kernel partition regular if and only if it satisfies the column condition. For some recent development on kernel partition regularity we refer to the articles \cite{new2, new1}.

To obtain an arithmetic progression of length \( m \), one considers the matrix
\[
A =
\begin{pmatrix}
    1 & 1 & -1 & 0 & \cdots & 0 \\
    1 & 2 & 0 & -1 & \cdots & 0 \\
    \vdots & \vdots & \vdots & \vdots & \ddots & \vdots \\
    1 & m & 0 & 0 & \cdots & -1 \\
\end{pmatrix},
\]
which corresponds to the statement of van der Waerden’s theorem. Another classical example is the matrix \( (1 \;\; 1 \;\; -1) \), which satisfies the column condition and encodes one of the earliest results in Ramsey theory: \emph{Schur's theorem}\cite{schur}. This theorem asserts that the pattern \( \{x, y, x + y\} \) is partition regular. By applying the map \( n \mapsto 2^n \), one immediately obtains the multiplicative analogue of Schur's theorem, which states that the pattern \( \{x, y, x \cdot y\} \) is partition regular. In \cite{comb}, Csikvári, Gyarmati, and Sárközy conjectured that the equation \( a + b = c \cdot d \) is partition regular. This conjecture was independently resolved by Bergelson and Hindman in \cite{pissa, h2}, using tools from the algebraic structure of the Stone–Čech compactification of discrete semigroups. In this article, we study polynomial generalizations of two classical results: Rado’s characterization of kernel partition regularity and the theorem asserting the partition regularity of the equation \( a + b = c \cdot d \).

The study of polynomial generalizations of classical results in Ramsey theory is inspired by \cite{2, 30}, where the authors established the \emph{polynomial van der Waerden theorem}. Let \( \mathbb{P} \) denote the set of polynomials with rational coefficients and zero constant term. The polynomial van der Waerden theorem states that for any finite subset \( F \subset \mathbb{P} \), and for any finite coloring of \( \mathbb{N} \), there exist \( a, d \in \mathbb{N} \) such that the set \( \{ a + P(d) : P \in F \} \) is monochromatic.

\medskip

The first goal of this article is to demonstrate the existence of certain nonlinear systems of equations that are partition regular. To achieve this, we extend the sum-equals-product theorem to a polynomial setting. In addition, we establish an infinitary version of the sum-equals-product theorem.
 Our approach involves a novel technique using a special class of ultrafilters that are simultaneously large with respect to both addition and multiplication\footnote{A detailed discussion of these ultrafilters is provided later.}. This method enables us to fuse additive polynomial structures with multiplicatively large sets, ultimately yielding the desired results.

 For example, one of our results shows that for any polynomial \( P \in \mathbb{P} \), the following system of equations is partition regular:
\[
\begin{aligned}
    x \cdot y &= z_1 + P(z) \\
    x \cdot y^2 &= z_2 + P(z)
\end{aligned}
\]

Another example of our approach reveals that the following equation is partition regular.

\[
\begin{aligned}
    x_1 + x_2 &= y_1\cdot z \\
    x_1 + 2x_2 &= y_2\cdot z.
\end{aligned}
\]

\medskip

To extend Rado's result in the nonlinear settings, we introduce a class of nonlinear systems.
\begin{defn}[\textbf{Nonlinear Rado System}]\label{newdef}
    Let \( A \) be an \( m \times n \) matrix (with \( n \geq 2 \)) with integer coefficients that satisfies the column condition, and let \( P_1, \ldots, P_m \) be elements of \( \mathbb{P} \), the set of rational-coefficient polynomials with zero constant term. Consider the system of equations:
    \begin{align*}
        a_{1,1}x_1 + \cdots + a_{1,n-1}x_{n-1} + a_{1,n} \cdot y_1 + P_1(z) &= 0 \\
        &\vdots \\
        a_{m,1}x_1 + \cdots + a_{m,n-1}x_{n-1} + a_{m,n} \cdot y_m + P_m(z) &= 0
    \end{align*}
    This system is called a \emph{Nonlinear Rado System}.
\end{defn}

Note that, in contrast to the linear case where one typically sets \( y_1 = \cdots = y_m \), here we allow the variables \( y_1, \ldots, y_m \) to be distinct. This distinction is necessary due to the nonlinear terms \( P_i(z) \). However, if all the polynomials \( P_i \) are identically zero, then the system reduces to the linear case, and we may take \( y_1 = \cdots = y_m \). We can summarize certain nonlinear Rado systems using a matrix framework. Given any matrix \( A = (a_{i,j}) \) of size \( m \times n \), we define the \emph{expanded matrix} of \( A \) as follows:
 \[
E(A) =
\begin{pmatrix}
    a_{1,1} & \cdots & a_{1,n-1} & a_{1,n} & 0 &\cdots & 0 \\
    a_{2,1} & \cdots & a_{2,n-1} & 0 &a_{2,n}& \cdots & 0 \\
    \vdots & \vdots & \vdots & \vdots & \vdots & \ddots & \vdots \\
    a_{m,1} & \cdots & a_{m,n-1} & 0 & 0 & \cdots & a_{m,n} \\
\end{pmatrix}.
\]
If \( P_1, P_2, \ldots, P_m \) are elements of \( \Po \), then we define the column matrix \( P \) as
\[
P = \begin{pmatrix}
    P_1\\
    P_2\\
    \vdots \\
    P_m
\end{pmatrix}.
\]
 Then, the system of equations described in Definition~\ref{newdef} can be expressed compactly as
\[
E(A)\vec{X} + P(z) = 0.
\]
Our second main result establishes that every \emph{nonlinear Rado system} is partition regular. In other words, if the matrix \( A \) satisfies the column condition, then the equation
\[
E(A)\vec{X} + P(z) = 0
\]
is partition regular.

\medskip

\subsection{Towards a Necessary Condition}

Our third major result concerns identifying a necessary condition under which a system of polynomial equations remains partition regular.

\medskip
\begin{example}

Before stating our main result, let us examine the following system of equations:
\begin{equation} \label{eq:nonlinear-system}
\begin{cases}
    F_1(x,y,z)=x-y+5z+z^2=0,\\
    F_2(x,y,z)=x+2y-3z=0.
\end{cases}
\tag{A}
\end{equation}
By the polynomial van der Waerden theorem, the first equation, and from Rado's theorem, the second equation is partition regular. However, our main theorem in this section shows that the system, when considered as a whole, fails to be partition regular.
\end{example}

\medskip

Let \( A = (a_{i,j}) \) be an \( m \times l \) matrix, and let \( P = \{P_1, \ldots, P_m\} \subset \mathbb{P} \) be a fixed finite set of polynomials. Let \( n \) denote the maximum degree among the polynomials in \( P \). For any \( z \in \mathbb{N} \), the vector \( P(z) \in \mathbb{Z}^m \) can be expressed as:
\[
P(z) =
\begin{pmatrix}
a_{1,l+1} \\
a_{2,l+1} \\
\vdots \\
a_{m,l+1}
\end{pmatrix} z +
\begin{pmatrix}
a_{1,l+2} \\
a_{2,l+2} \\
\vdots \\
a_{m,l+2}
\end{pmatrix} z^2 +
\cdots +
\begin{pmatrix}
a_{1,l+n} \\
a_{2,l+n} \\
\vdots \\
a_{m,l+n}
\end{pmatrix} z^n.
\]
Here, for each \( 1 \leq i \leq m \), the polynomial \( P_i(z) \) takes the form
\[
P_i(z) = a_{i,l+1} z + a_{i,l+2} z^2 + \cdots + a_{i,l+n} z^n,
\]
where some of the coefficients may be zero.

Define
\[
C = \begin{pmatrix}
a_{1,l+1} & a_{1,l+2} & \cdots & a_{1,l+n} \\
a_{2,l+1} & a_{2,l+2} & \cdots & a_{2,l+n} \\
\vdots & \vdots & \cdots & \vdots \\
a_{m,l+1} & a_{m,l+2} & \cdots & a_{m,l+n}
\end{pmatrix}
\]
to be the coefficient matrix associated with the polynomial vector \( P \), and let
\[
A_{\text{aug}}(P) = \begin{pmatrix} A & C \end{pmatrix}
\]
be the augmented matrix formed by adjoining \( A \) with the coefficients of \( P \).

Define the submatrix
\[
A_{\text{aug}}^{\text{lin}}(P) =
\begin{pmatrix}
a_{1,1} & a_{1,2} & \cdots & a_{1,l+1} \\
a_{2,1} & a_{2,2} & \cdots & a_{2,l+1} \\
\vdots & \vdots & \cdots & \vdots \\
a_{m,1} & a_{m,2} & \cdots & a_{m,l+1}
\end{pmatrix}
\]
to be the part of \( A_{\text{aug}}(P) \) corresponding to the linear terms of the polynomials in \( P \).

\noindent$\textbf{Note}.$ Observe that the augmented matrix corresponding to the system of equations \eqref{eq:nonlinear-system} is
\[
A_{\text{aug}} = \begin{pmatrix}
    1 & -1 & 5 &1\\
    1 & 2 & -3&0
\end{pmatrix}.
\]
However, the matrix $A_{\text{aug}}^{\text{lin}}$ fails to satisfy the first condition of Theorem \ref{N1}. Consequently, the system is not partition regular, in the sense that there exists no constant monochromaticromatic solution.

The following theorem provides a necessary condition for the partition regularity of the nonlinear Rado system.
Before stating our main result, define \( c_{t_0} = \vec{0} \), and for each \( 1 \leq j \leq n-1 \), let \( c_{t_j} = c_{l+j+1} \).

\begin{thm}\label{N2}
Let \( P = \{P_1, \ldots, P_m\} \subset \mathbb{P} \) be a finite set of polynomials, and let \( A = (a_{ij}) \) be an \( m \times l \) matrix with rational entries. Suppose that for every finite coloring of \( \mathbb{N} \), there exists a nontrivial monochromatic set (no constant solution) \( \vec{X} = \{x_1, \ldots, x_l\} \cup \{z\} \subset \mathbb{N} \) such that the equation
\[
A\vec{X} + P(z) = \vec{0}
\]
is satisfied. Then there exists \( s \in \mathbb{N} \) and a partition $[1,l+1]=\cup_{j=1}^{s}I_j$
 such that the following conditions hold:


\begin{enumerate}
\item Assume that the column \( c_{l+1} \) is not the zero vector. Then there exists a subset \( I_1 \subseteq \{1, 2, \ldots, l+1\} \) such that
\[
\sum_{i \in I_1} c_i = \vec{0},
\]

However if we assume that $c_{n+1}=\vec 0,$ then choose $n+i=\min\{n+j:c_{n+j}\neq \vec 0\}.$ Then our technique says that 

\begin{itemize}
    \item[(1')] there exists $I_1\subseteq \{1, 2, \ldots, l\}$ such that \(\vec 0\in \operatorname{span}_{\mathbb{Q}}\left(\left\lbrace c_p:p\in I_1\cup \{{n+i}\}\right\rbrace \right)\).
\end{itemize}

    \item For each \( 1 \leq j \leq s \), there exists \( 0 \leq k \leq n-1 \) such that
    \[
  \sum_{i\in I_j}c_i   \in \operatorname{span}_{\mathbb{Q}} \left( \{ c_p : p \in \cup_{k=1}^{j-1}I_k \} \cup \{ c_{t_i} : 0 \leq i \leq k \} \right).
    \]

    \item There exists \( 0 \leq q \leq s \) such that
    \[
    c_{l+2} \in \operatorname{span}_{\mathbb{Q}} \{ c_i : i \in \cup_{k=1}^{q}I_k \}.
    \]

    \item For each \( 3 \leq u \leq n \), there exists \( 0 \leq v \leq s \) such that
    \[
    c_{l+u} \in \operatorname{span}_{\mathbb{Q}} \left( \{ c_i : i\in \cup_{k=1}^{v}I_k \} \cup \{ c_{l+2}, \ldots, c_{l+u-1} \} \right).
    \]
\end{enumerate}

Note that the conditions $(3),(4)$ together reduces to the following condition.
\begin{enumerate}
    \item[(4')] For each \( 2 \leq u \leq n \), there exists \( 0 \leq v' \leq s \) such that
    \[
    c_{l+u} \in \operatorname{span}_{\mathbb{Q}} \left( \{ c_i : i\in \cup_{k=1}^{v'}I_k \}  \right).
    \]
\end{enumerate}

Moreover conditions $2,4'$, together implies that the linear part \( A_{\text{aug}}^{\text{lin}}(P) \) satisfies the condition:

   \begin{enumerate}
       \item[(2')] for each $2\leq j\leq s,$ \[
  \sum_{i\in I_j}c_i   \in \operatorname{span}_{\mathbb{Q}} \left( \left\lbrace c_p : p \in \cup_{k=1}^{s}I_k\setminus I_j \right\rbrace \right).
    \]
     \end{enumerate}
\end{thm}

\begin{rem}\text{}
    
\begin{itemize}
    \item The above condition (4') implies that each column of the augmented matrix \( A_{\text{aug}}^{\text{poly}}(P) \) is a linear combination (over \( \mathbb{Q} \)) of certain columns of the linear part \( A_{\text{aug}}^{\text{lin}}(P) \).

    
    \item Note that if each \( P_i \) is chosen to be the zero polynomial, then the theorem immediately reduces to the classical column condition of Rado.
    
    \item  Furthermore, if each \( P_i \) is a linear polynomial with zero constant term, i.e., \( P_i(z) = c_i z \) for some constants \( c_i \), then the system reduces to the linear form \( A_{\text{aug}}^{\text{lin}}(P)\vec{X} = \vec{0} \), and the result asserts that \( A_{\text{aug}}^{\text{lin}}(P) \) satisfies the column condition.

\end{itemize}
\end{rem}

\begin{note}
    By multiplying the entire system \( A\vec{X} + P(z) = \vec{0} \) by a sufficiently large integer, one can ensure that all entries of the matrix \( A \) and all coefficients of the polynomials \( P_i(z) \) become integers. Therefore, it suffices to prove the above theorem under the assumption that both the matrix \( A \) and the coefficients of \( P(z) \) have integer entries.
\end{note}

Following the proof of Theorem~\ref{N2}, we derive a necessary condition that must be satisfied for a system of nonlinear inhomogeneous equations to be partition regular.


\begin{note}
In \cite{update}, Bergelson posed the question of whether the equation \( x - 2y = P(z) \), where \( P \in \mathbb{P} \) is non-linear, is partition regular. Utilizing \cite[Example~3.15]{NB}, Di Nasso and Luperi Baglini proved that this equation is not partition regular. 

However, our combinatorial method shows this conclusion with four exceptional cases. According to Theorem~\ref{N2}, if \( n > 1 \), the partition regularity of the equation
\[
\sum_{j=1}^{m} a_j x_j + \sum_{i=1}^{n} a_{m+i} z^i = 0
\]
requires the existence of a nonempty subset \( J \subseteq \{1, 2, \dots, n+1\} \) such that \( \sum_{j \in J} c_j = 0 \). 

In particular, for the case \( x - 2y + P(z) = 0 \), this implies that the equation is not partition regular unless \( a_{m+1} \notin \{-1,0, 1, 2\} \). 

\end{note}
\begin{note}
In a recent article \cite{G25}, the author established that the equation \( z - x - x y = 0 \) is partition regular. This example highlights that in multivariable settings, the presence of nonlinear terms does not necessarily exclude linear contributions from playing a significant role. Thus, the linear case must be handled with care, even in nonlinear contexts.

Building upon the techniques developed in the proof of Theorem~\ref{N2}, we derive the following necessary condition for a broader class of nonlinear Diophantine equations.

\begin{thm}\label{justadd}
Let \( m\in \mathbb{N} \), and let
\[
P(y_1, \ldots, y_m) = \sum_{i=1}^{m} c_{i} y_i + Q(y_1, \ldots, y_m)
\]
be a polynomial in \( \mathbb{Z}[y_1, \ldots, y_m] \), where all $c_i's$ are not zero simultaneously, and polynomial \( Q \) has the following structure:
\begin{itemize}
    \item Let \( I \subseteq \{1, 2, \ldots, m\} \) denote the support of the coefficients \( c_1, c_2, \ldots, c_m \); that is,
    \[
    I = \{ i \in \{1, \ldots, m\} \mid c_i \neq 0 \};
    \]
    
    \item The polynomial \( Q \) consists only of monomials of the form \( y_1^{i_1} \cdots y_m^{i_m} \), where \( i_j > 0 \) for some \( j \in I \).
\end{itemize}

 If the equation
\[
 P(y_1, \ldots, y_n) = 0
\]
is partition regular, then there exists a nonempty subset \( I \subseteq \{1, 2, \ldots, n\} \) such that
\[
\sum_{i \in I} c_i = 0.
\]
\end{thm}
\begin{proof}
    Same as Theorem~\ref{N2}.
\end{proof}
\end{note}

So our above Theorem suggest that the following equation is not partition regular.
\begin{thm}
Let \( c \in \mathbb{Z} \setminus \{0, -1\} \), and let \( P(y) \in \mathbb{Z}[y] \) be a polynomial with no constant terms. Then for any \(n\in \N,\) the Diophantine equation
\[
z + c x +x^n P(y) = 0
\]
is not partition regular over \( \mathbb{N} \).
\end{thm}

\begin{rem}
Although from \cite{anal, G25} it is known that systems such as 
\[
\{x,\, xy,\, x + y\} \quad \text{and} \quad \{x,\, y,\, xy,\, xy + y\}
\]
admit monochromatic solutions under any finite coloring of \( \mathbb{N} \), our previous theorem shows that this behavior does not extend universally to all nonlinear combinations. In particular, the set 
\[
\{x,\, y,\, xy - y\}
\] 
is \emph{not} partition regular.
\end{rem}


In \cite{BLM}, Barrett, Lupini, and Moreira introduced the notion of the \emph{maximal Rado condition} and established a necessary condition for a single nonlinear equation. All the previous works rely on tools from nonstandard analysis. In contrast, our result provides a necessary condition for an entire system of nonlinear equations using purely combinatorial arguments.

Our proofs of sufficient conditions rely on the machinery of ultrafilters. Before proceeding to the main results, we briefly recall some fundamental concepts related to ultrafilters and their connection to Ramsey theory. Recall some fundamental concepts related to ultrafilters and their connection to Ramsey theory.

\section{Preliminaries of Ultrafilters}

 Ultrafilters are set-theoretic objects that play a central role in modern Ramsey theory. In this section, we give a brief overview of the basic notions. For a comprehensive treatment, we refer the reader to the book by Hindman and Strauss~\cite{book}.

A \emph{filter} \( \mathcal{F} \) on a nonempty set \( X \) is a collection of subsets of \( X \) satisfying the following properties:
\begin{enumerate}
    \item \( \emptyset \notin \mathcal{F} \) and \( X \in \mathcal{F} \),
    \item If \( A \in \mathcal{F} \) and \( A \subseteq B \subseteq X \), then \( B \in \mathcal{F} \),
    \item If \( A, B \in \mathcal{F} \), then \( A \cap B \in \mathcal{F} \).
\end{enumerate}

Using Zorn's Lemma, one can show that every filter is contained in a maximal filter, called an \emph{ultrafilter}. Any ultrafilter \( p \) on \( X \) satisfies the following partition property:
\begin{itemize}
    \item For every finite partition \( X = A_1 \cup \cdots \cup A_r \), there exists \( i \in \{1, 2, \ldots, r\} \) such that \( A_i \in p \).
\end{itemize}

Let \( S \) be a discrete semigroup. The elements of \( \beta S \), the Stone-\v{C}ech compactification of \( S \), are regarded as ultrafilters on \( S \). For any subset \( A \subseteq S \), we define the set \( \overline{A} = \{ p \in \beta S : A \in p \} \). The collection \( \{\overline{A} : A \subseteq S \} \) forms a basis for the closed sets in \( \beta S \). The semigroup operation \( \cdot \) on \( S \) extends to the Stone-\v{C}ech compactification \( \beta S \), making \( (\beta S, \cdot) \) a compact right topological semigroup. This means that for each \( p \in \beta S \), the function \( \rho_p(q) : \beta S \to \beta S \), defined by \( \rho_p(q) = q \cdot p \), is continuous, and \( S \) is contained within the topological center of \( \beta S \). Specifically, for any \( x \in S \), the function \( \lambda_x : \beta S \to \beta S \), defined by \( \lambda_x(q) = x \cdot q \), is also continuous.

This structure is connected to a famous result by Ellis, which states that if \( S \) is a compact right topological semigroup, then the set of idempotents \( E(S) \neq \emptyset \). In semigroup theory, a nonempty subset \( I \) of a semigroup \( T \) is called a \textbf{left ideal} of \( S \) if \( T I \subset I \), a \textbf{right ideal} if \( I T \subset I \), and a \textbf{two-sided ideal} (or simply an \textbf{ideal}) if it is both a left and right ideal. A \textbf{minimal left ideal} is a left ideal that does not contain any proper left ideal, and similarly, we can define a minimal right ideal and the smallest ideal.

For any compact Hausdorff right topological semigroup \( T \), the smallest two-sided ideal, denoted \( K(T) \), is given by:

\[
K(T) = \bigcup \{ L : L \text{ is a minimal left ideal of } T \} = \bigcup \{ R : R \text{ is a minimal right ideal of } T \}.
\]

Given a minimal left ideal \( L \) and a minimal right ideal \( R \), their intersection \( L \cap R \) is a group, and in particular, it contains an idempotent. If \( p \) and \( q \) are idempotents in \( T \), we write \( p \leq q \) if and only if \( pq = qp = p \). An idempotent is minimal with respect to this relation if and only if it is a member of the smallest ideal \( K(T) \).

For \( p, q \in \beta S \) and \( A \subseteq S \), it holds that \( A \in p \cdot q \) if and only if the set \( \{ x \in S : x^{-1}A \in q \} \in p \), where \( x^{-1}A = \{ y \in S : x \cdot y \in A \} \).

 \begin{center}
$\bullet$    A set $A\subseteq S$ is said to be a \emph{piecewise syndetic} if there exists $p\in K(\beta S,\cdot)$ such that $A\in p.$
\end{center}

Ellis' Theorem asserts that every compact right topological semigroup contains an idempotent. If \( T \) is a compact right topological semigroup, we denote by \( E(T) \) the set of all idempotents of \( T \), i.e., 
\[
E(T) := \{ e \in T : e \cdot e = e \}.
\]

The existence of idempotents in \( \beta \mathbb{N} \) is intimately related to \textbf{Hindman's Theorem} \cite{h}. For any nonempty set \( X \), let \( \mathcal{P}_f(X) \) denote the collection of all nonempty finite subsets of \( X \). For any commutative semigroup \( (S, +) \), a set \( A \subseteq S \) is said to be an \emph{IP set} if there exists a sequence \( \langle x_n \rangle_n \) in \( S \) such that
\[
A = \operatorname{FS}(\langle x_n \rangle_n) = \left\{ \sum_{t \in H} x_t : H \in \mathcal{P}_f(\mathbb{N}) \right\}.
\]
Similarly, for any \( r \in \mathbb{N} \), a set \( A \subseteq S \) is called an \emph{IP\(_r\)} set if there exists a sequence \( \langle x_n \rangle_{n=1}^r \) such that
\[
A = \operatorname{FS}(\langle x_n \rangle_{n=1}^r) = \left\{ \sum_{t \in H} x_t : \emptyset \neq H \subseteq \{1, 2, \ldots, r\} \right\}.
\]

From the ultrafilter proof of Hindman's Theorem\footnote{This is an unpublished result due to Galvin and Glazer}, it follows that if \( A \in p \) for some idempotent ultrafilter \( p \), then \( A \subseteq S \) is an IP set. Moreover, this condition is also \emph{sufficient}; that is, \( A \) is an IP set if and only if \( A \in p \) for some idempotent \( p \in \beta S \).

\begin{defn}
     A set $A\subseteq S$ is \emph{central} if there exists $p\in E\left(K(\beta S,\cdot)\right)$ such that $A\in p.$
\end{defn}
A set \( A \subseteq \mathbb{N} \) is called \emph{additively large} (respectively, \emph{multiplicatively large}) if it is large in the semigroup \( (\mathbb{N}, +) \) (respectively, \( (\mathbb{N}, \cdot) \)).

From \cite{2, 30}, it follows that every additively piecewise syndetic set witnesses the conclusion of the polynomial van der Waerden theorem.

It is a standard exercise to verify that the set \( \overline{E(\beta \mathbb{N}, +)} \), the closure of the idempotents in \( (\beta \mathbb{N}, +) \), forms a left ideal in the semigroup \( (\beta \mathbb{N}, \cdot) \). Consequently, the intersection 
\[
\overline{E(\beta \mathbb{N}, +)} \cap E(\beta \mathbb{N}, \cdot)
\]
is nonempty.

Similarly, the closure of the minimal ideal \( \overline{E(K(\beta \mathbb{N}, +))} \) is a left ideal in \( (\beta \mathbb{N}, \cdot) \), which implies that
\[
\overline{E(K(\beta \mathbb{N}, +))} \cap E(K(\beta \mathbb{N}, \cdot))
\]
is also nonempty.

Any ultrafilter \( p \) in this intersection is called a \emph{centrally rich ultrafilter}. If \( p \) is centrally rich and \( A \in p \), then \( A \) is both additively and multiplicatively central.

Before we proceed to our main proof, let us recall one final interesting fact about ultrafilters, which we will use in our proof. A set \( A \subseteq \mathbb{N} \) is called an \emph{\(IP_0\) set} if, for every \( r \in \mathbb{N} \), the set \( A \) contains an \(IP_r\) set. Clearly, every member of an additive idempotent ultrafilter is an \(IP_0\) set.

We denote by \( \overline{\mathrm{IP}_0} \) the set of all ultrafilters in \( \beta \mathbb{N} \) whose every member is an \(IP_0\) set. It is a routine exercise to verify that \( \overline{\mathrm{IP}_0} \) is a two-sided ideal in the semigroup \( (\beta \mathbb{N}, \cdot) \). In particular, this implies that
\[
\overline{K(\beta \mathbb{N}, \cdot)} \subseteq \overline{\mathrm{IP}_0}.
\]

\section{Proof of Sufficient Conditions}

\subsection{Extending the Sum-Equals-Product Theorem}

In this section, we focus on extending the classical Sum-Equals-Product phenomenon in Ramsey theory. We explore both infinitary and polynomial extensions of this idea, showing that even complex combinations of sums and products can exhibit partition regularity. 

\subsubsection{Infinitary Extension of the Sum-Equals-Product Theorem}

We begin by establishing an infinitary version of the Sum-Equals-Product phenomenon. Our goal is to construct sequences whose finite sums and finite products, as well as certain structured combinations of the two, lie entirely within a single cell of any finite coloring of the natural numbers. This result significantly strengthens earlier partition regularity results and lays the foundation for subsequent polynomial generalizations.

\begin{thm}\label{task2}
    Let \( r \in \mathbb{N} \), and suppose that \( \mathbb{N} = \bigcup_{i=1}^r A_i \) is a finite partition. Then there exist two sequences \( \langle a_n \rangle_n \) and \( \langle b_n \rangle_n \) in \( \mathbb{N} \), and an index \( i \in \{1,2,\ldots, r\} \), such that:
    \begin{enumerate}
        \item \( FS\left(\langle a_n \rangle_n\right) \cup FP\left(\langle b_n \rangle_n\right) \subseteq A_i \), and 
        \item For every \( N \in \mathbb{N} \), we have
        \[
        \bigcup_{m=1}^N FS\left(\langle a_n \rangle_{n=1}^m\right) \cdot FP\left(\langle b_n \rangle_{n=m}^N\right) \subseteq A_i.
        \]
    \end{enumerate}
\end{thm}

\begin{itemize}
    \item Since \( (a_1 + a_2) \cdot b_2 = a_1 \cdot b_2 + a_2 \cdot b_2 \), the theorem above implies that the equation \( a + b = c \cdot d \) is partition regular. That is, no matter how the natural numbers are finitely colored, there exists a monochromatic solution to the equation \( a + b = c \cdot d \).
    
    \item Also, observe that
    \[
    (a_1 + a_2 + \cdots + a_n) \cdot b_{n+1} b_{n+2} \cdots b_{n+(m-1)} = \sum_{i=1}^n a_i \cdot b_{n+1} b_{n+2} \cdots b_{n+(m-1)}.
    \]
    This shows that, for any \( m,n \in \mathbb{N} \), the equation
    \[
    x_1 + x_2 + \cdots + x_n = y_1 \cdot y_2 \cdots y_m
    \]
    is partition regular, as proved in \cite{h2}.
\end{itemize}

\begin{proof}[Proof of Theorem \ref{task2}]
Choose any 
\[
p \in \overline{E\left(\beta \mathbb{N},+\right)} \cap E\left(\beta \mathbb{N},\cdot\right),
\]
and let \( A \in p = p \cdot p \). Define:
\[
A^\star = \{x \in A : x^{-1}A \in p\} \in p.
\]
Then \( A^\star \) is an additive \( \text{IP} \)-set. Hence, there exists \( q \in E(\beta \mathbb{N}, +) \) such that \( A^\star \in q \).

Now define:
\[
A^{\star \star} = \{x \in A^\star : -x + A^\star \in q \} \in q.
\]
Then, by \cite[Lemma 4.14]{book}, for all \( x \in A^{\star \star} \), we have:
\[
-x + A^{\star \star} \subseteq A^{\star \star}.
\]

Choose \( a_1 \in A^{\star \star} \). Since \( A^{\star \star} \subseteq A^\star \), we also have \( a_1 \in A^\star \). 
Then the intersection \( A^\star \cap a_1^{-1}A^\star \in p \), because both sets belong to \( p \), and \( p \) is closed under finite intersections.

Now choose:
\[
b_1 \in A^\star \cap a_1^{-1}A^\star.
\]
Then:
\[
a_1 \in A^\star, \quad b_1 \in A^\star, \quad \text{and} \quad a_1 \cdot b_1 \in A^\star.
\]
So,
\[
\{a_1, b_1, a_1 \cdot b_1\} \subseteq A^\star \subseteq A.
\]

Now, since \( a_1 \in A^{\star \star} \), and by Lemma 4.14 from \cite{book}, we have:
\[
A^{\star \star} \cap (-a_1 + A^{\star \star}) \in q.
\]
So we choose:
\[
a_2 \in A^{\star \star} \cap (-a_1 + A^{\star \star}).
\]
This implies:
\[
\{a_1, a_2, a_1 + a_2\} \subseteq A^{\star \star} \subseteq A^\star.
\]

Now define:
\[
b_2 \in A^\star \cap \bigcap_{x \in \{a_1, a_2, a_1 + a_2\}} x^{-1}A^\star \cap \bigcap_{y \in \{a_1, b_1, a_1 \cdot b_1\}} y^{-1}A^\star.
\]
Then:
\[
\{b_2, a_1 \cdot b_2, a_2 \cdot b_2, (a_1 + a_2) \cdot b_2, a_1 \cdot b_1 \cdot b_2\} \subseteq A^\star \subseteq A.
\]

Hence, we have the following inclusions:

\begin{enumerate}
    \item \( \{a_1, a_2, a_1 + a_2\} \subset A^{\star \star} \subseteq A^\star, \)
    \item \( \{b_2\} \cup b_2 \cdot \{a_1, a_2, a_1 + a_2\} \subseteq A^\star, \)
    \item \( \{a_1, b_1, a_1 \cdot b_1\} \subset A^\star, \)
    \item \( b_2 \cdot \{a_1, b_1, a_1 \cdot b_1\} \subset A^\star. \)
\end{enumerate}

Inductively, we assume that for some \( N \in \mathbb{N} \), we have two sequences \( \langle a_n \rangle_{n=1}^N \) and \( \langle b_n \rangle_{n=1}^N \) such that:
\begin{enumerate}
    \item \( FS\left(\langle a_n \rangle_{n=1}^N\right) \subset A^{\star\star}, \)
    \item \( FP\left(\langle b_n \rangle_{n=1}^N\right) \subset A^\star, \)
    \item For every \( M \leq N \), we have
    \[
    \bigcup_{m=1}^M FS\left(\langle a_n \rangle_{n=1}^m\right) \cdot FP\left(\langle b_n \rangle_{n=m}^M\right) \subset A^\star.
    \]
\end{enumerate}

To complete the induction, choose
\[
a_{N+1} \in A^{\star\star} \cap \bigcap_{x \in FS\left(\langle a_n \rangle_{n=1}^N\right)} (-x + A^{\star\star}) \in q.
\]
Hence,
\[
FS\left(\langle a_n \rangle_{n=1}^{N+1}\right) \subset A^{\star\star}.
\]

Define
\[
B = \bigcup_{M=1}^N \bigcup_{m=1}^M FS\left(\langle a_n \rangle_{n=1}^m\right) \cdot FP\left(\langle b_n \rangle_{n=m}^M\right) \subset A^\star.
\]

Now choose
\[
b_{N+1} \in A^\star \cap \bigcap_{x \in FS\left(\langle a_n \rangle_{n=1}^N\right)} x^{-1}A^\star \cap \bigcap_{y \in B} y^{-1}A^\star \cap \bigcap_{z \in FP\left(\langle b_n \rangle_{n=1}^N\right)} z^{-1}A^\star \in p.
\]

Now clearly we have:
\begin{enumerate}
    \item \( FS\left(\langle a_n \rangle_{n=1}^{N+1} \right) \subset A^{\star\star} \),
    \item \( FP\left(\langle b_n \rangle_{n=1}^{N+1} \right) \subset A^\star \), and
    \item \( FS\left(\langle a_n \rangle_{n=1}^{N+1} \right) \cdot b_{N+1} \cup B \cdot b_{N+1} \subseteq A^\star \).
\end{enumerate}

Now, from the definition of \( B \), it is clear that this verifies the third inductive hypothesis for \( N+1 \) as well.

This completes the proof.
\end{proof}

However the technique of the proof of the above theorem can be generalized further to prove a more general result. Before we describe that let us introduce two notions of regular structure.
\begin{defn}\text{}\label{defn}
    \begin{enumerate}
    \item \textbf{(Additive regular):} A family of finite subsets $\F$ of $\N$ is called \textit{additive regular} if every additive central set contains an element $F\in \F.$

    \item \textbf{(Multiplicative regular):} A family of finite subsets $\G$ of $\N$ is called \textit{multiplicative regular} if every multiplicative central set contains an element $G\in \G.$
\end{enumerate}
\end{defn}

If we closely look at the proof of  Theorem \ref{task2}, we see that one can replace each $a_n$ by an element of a given additive regular family, and each $b_n$ by an element of a given multiplicative regular family. 

For any two finite sets $F,G$, define 
\[
F \cdot G = \{ f \cdot g : f \in F,\, g \in G \}.
\]

Given a sequence of finite sets $\langle F_n \rangle_{n=1}^N$, define:
\begin{itemize}
    \item $FS\left( \langle F_n \rangle_{n=1}^N \right) = \bigcup_{f_i \in F_i \text{ for each } i \in \{1, \ldots, N\}} FS\left( (f_i)_{i=1}^N \right),$
    \item $FP\left( \langle F_n \rangle_{n=1}^N \right) = \bigcup_{f_i \in F_i \text{ for each } i \in \{1, \ldots, N\}} FP\left( (f_i)_{i=1}^N \right).$
\end{itemize}

The proof of the following theorem is similar to the proof of Theorem \ref{task2}, and hence we omit it.

\begin{thm}\label{notask}
Let $\langle \mathcal{F}_n \rangle_n$ and $\langle \mathcal{G}_n \rangle_n$ be sequences of additive and multiplicative regular families, respectively.

Then for any \( r \in \mathbb{N} \) and any finite partition \( \mathbb{N} = \bigcup_{i=1}^r A_i \), there exist sequences \( \langle F_n \rangle_n \) and \( \langle G_n \rangle_n \) such that:
\begin{enumerate}
    \item For every \( n \in \mathbb{N} \), we have \( F_n \in \mathcal{F}_n \) and \( G_n \in \mathcal{G}_n \);
    \item \( FS\left( \langle F_n \rangle_n \right) \cup FP\left( \langle G_n \rangle_n \right) \subseteq A_i \) for some \( i \in \{1,2,\ldots,r\} \);
    \item For every \( N \in \mathbb{N} \),
    \[
    \bigcup_{m=1}^N FS\left( \langle F_n \rangle_{n=1}^m \right) \cdot FP\left( \langle G_n \rangle_{n=m}^N \right) \subseteq A_i.
    \]
\end{enumerate}
\end{thm}


\begin{cor}
    For any \( l, m, n \in \mathbb{N} \) with \( n > 2 \), \( m > 1 \), the system of nonlinear equations
    \begin{align*}
        x_1 + x_2 + x_3 + \cdots + x_n &= z_1 \cdot y_1 y_2 \cdots y_m, \\
        x_1 + 2 x_2 + x_3 + \cdots + x_n &= z_2 \cdot y_1 y_2 \cdots y_m, \\
        &\vdots \\
        x_1 + l x_2 + x_3 + \cdots + x_n &= z_l \cdot y_1 y_2 \cdots y_m,
    \end{align*}
    is partition regular.
\end{cor}

\begin{proof}
    Define the families:
    \begin{itemize}
        \item \( \mathcal{F}_1 = \{ \{a, a + d, a + 2d, \dots, a + ld\} : a, d \in \mathbb{N} \} \),
        \item \( \mathcal{F}_2 = \left\{ \left\{ b_1 + \cdots + b_{n-2} \right\} : b_i \in \mathbb{N} \text{ for all } i \right\} \),
        \item \( \mathcal{G}_2 = \left\{ \left\{ c_1 \cdots c_{m-1} \right\} : c_i \in \mathbb{N} \right\} \).
    \end{itemize}

    By Theorem \ref{notask}, there exist sets \( F_1 \in \mathcal{F}_1 \), \( F_2 \in \mathcal{F}_2 \), and \( G_2 \in \mathcal{G}_2 \) such that \( (F_1 + F_2) \cdot G_2 \) is monochromatic.

    Then, for each \( i \in \{1, \dots, l\} \), the quantity
    \[
    (a + i d + b_1 + \cdots + b_{n-2}) \cdot c_1 \cdots c_{m-1}
    \]
    belongs to \( (F_1 + F_2) \cdot G_2 \), and can be rewritten as
    \[
    x_1 + i x_2 + x_3 + \cdots + x_n = y_1 \cdots y_m.
    \]
    This gives a monochromatic solution for each row in the system. Hence, the system is partition regular.
\end{proof}

In the above proof, we may choose \( \mathcal{G}_2 \) to be the set of geometric progressions of length \( m \) to obtain the following corollary. The proof is similar and thus omitted.

\begin{cor}
    For any \( l, m, n \in \mathbb{N} \) with \( n > 2 \) and \( m > 1 \), the system of nonlinear equations
    \begin{align*}
        x_1 + x_2 + x_3 + \cdots + x_n &= y_1 \cdot z^m, \\
        x_1 + 2 x_2 + x_3 + \cdots + x_n &= y_2 \cdot z^m, \\
        &\vdots \\
        x_1 + l x_2 + x_3 + \cdots + x_n &= y_l \cdot z^m,
    \end{align*}
    is partition regular.
\end{cor}

\subsubsection{Polynomial Extension of Sum-Equals-Product Theorems}
We now extend the sum-equals-product phenomenon to a broader class of nonlinear systems involving polynomial perturbations. Specifically, we show that even when additive expressions are equated to products modified by polynomials, the system remains partition regular.

\begin{thm}[\textbf{Polynomial Sum-Equals-Product Theorem}]\label{task3}
    Let \( m, n, r \in \mathbb{N} \), and let \( \{P_i : i = 1, 2, \ldots, r\} \subset \Po \) be a finite collection of polynomials with zero constant term. Then the following system of nonlinear equations is partition regular:
    \begin{align*}
        x_1 + x_2 + \cdots + x_n &= y_1 y_2 \cdots y_m \cdot z_1 + P_1(y_{m+1}), \\
        x_1 + x_2 + \cdots + x_n &= y_1 y_2 \cdots y_m \cdot z_2 + P_2(y_{m+1}), \\
        &\vdots \\
        x_1 + x_2 + \cdots + x_n &= y_1 y_2 \cdots y_m \cdot z_r + P_r(y_{m+1}).
    \end{align*}
\end{thm}


   \begin{proof}
We begin by assuming that \( m \geq 2 \).

Let \( p \) be a centrally rich ultrafilter and let \( A \in p \). Since \( p = p \cdot p \), it follows that the multiplicative star set
\[
A^\star := \{x \in A : x^{-1}A \in p\}
\]
also belongs to \( p \).

A well-known fact from ultrafilter theory asserts that
\[
\overline{K(\beta \mathbb{N}, \cdot)} \subseteq \overline{\mathrm{IP}_0},
\]
so every member of \( p \) contains finite products of some IP-sequence. In particular, since \( A^\star \) is multiplicatively piecewise syndetic, there exists a finite sequence \( \langle a_i \rangle_{i=1}^n \subset \mathbb{N} \) such that
\[
FS\left(\langle a_i \rangle_{i=1}^n\right) \subset A^\star.
\]

Define the set
\[
B := A^\star \cap \bigcap_{x \in FS\left(\langle a_i \rangle_{i=1}^n\right)} x^{-1}A^\star.
\]
Then \( B \in p \), and so \( B^\star \in p \) as well.

Now choose a finite sequence \( \langle b_j \rangle_{j=1}^{m-1} \) such that
\[
FP\left(\langle b_j \rangle_{j=1}^{m-1}\right) \subset B^\star.
\]

Define
\[
C := B^\star \cap \bigcap_{y \in FP\left(\langle b_j \rangle_{j=1}^{m-1}\right)} y^{-1}B^\star.
\]
Since \( C \in p \) and \( p \) is centrally rich, it follows that \( C \) is an additively central set.

 By the polynomial van der Waerden theorem, we can choose \( a, d \in \mathbb{N} \) such that
\[
\left\{ a,\, d,\, a - \frac{1}{(a_1 + \cdots + a_n) \cdot b_1 \cdots b_{m-1}} \cdot P_i(d) : i = 1, 2, \ldots, r \right\} \subset C.
\]
Thus,
\[
\left\{(a_1 + \cdots + a_n) \cdot b_1 \cdots b_{m-1} \cdot \left( a - \frac{1}{(a_1 + \cdots + a_n) \cdot b_1 \cdots b_{m-1}} \cdot P_i(d) \right) : i = 1, 2, \ldots, r \right\} \subset A^\star.
\]

Now define:
\begin{enumerate}
    \item For each \( i \in \{1, \ldots, r\} \),
    \[
    z_i := a - \frac{1}{(a_1 + \cdots + a_n) \cdot b_1 \cdots b_{m-1}} \cdot P_i(d),
    \]
    \item For each \( j \in \{1, \ldots, n\} \),
    \[
    x_j := a_j \cdot b_1 \cdots b_{m-1} \cdot a,
    \]
    \item For each \( k \in \{1, \ldots, m\} \),
    \begin{itemize}
        \item If \( k \neq m \), then \( y_k := b_k \),
        \item If \( k = m \), then \( y_m := a_1 + \cdots + a_n \),
    \end{itemize}
    \item Finally, let \( y_{m+1} := d \).
\end{enumerate}

 Now, for each \( i \in \{1, \ldots, r\} \), we compute:
\begin{align}
&(a_1 + \cdots + a_n)\cdot b_1 \cdots b_{m-1} \cdot \left( a - \frac{1}{(a_1 + \cdots + a_n)\cdot b_1 \cdots b_{m-1}} \cdot P_i(d) \right) \nonumber \\
&= (a_1 + \cdots + a_n)\cdot b_1 \cdots b_{m-1} \cdot a - P_i(d) \nonumber \\
&= x_1 + \cdots + x_n - P_i(y_{m+1}). \label{1}
\end{align}

But again,
\begin{align}
&(a_1 + \cdots + a_n)\cdot b_1 \cdots b_{m-1} \cdot \left( a - \frac{1}{(a_1 + \cdots + a_n)\cdot b_1 \cdots b_{m-1}} \cdot P_i(d) \right) \nonumber \\
&= (a_1 + \cdots + a_n)\cdot b_1 \cdots b_{m-1} \cdot z_i \nonumber \\
&= y_1 \cdots y_m \cdot z_i. \label{2}
\end{align}

Equating expressions \eqref{1} and \eqref{2}, we conclude that:
\[
x_1 + \cdots + x_n = y_1 \cdots y_m \cdot z_i + P_i(y_{m+1}),
\]
as required.

   Now for each \( i \in \{1,\ldots, r\} \), from equations \eqref{1} and \eqref{2}, we have
\[
x_1 + \cdots + x_n = y_1 \cdots y_m \cdot z_i + P_i(y_{m+1}).
\]

For the case \( m = 1 \), we interpret the product \( b_1 \cdots b_{m-1} \) as the empty product, i.e., equal to 1. Then the entire construction and argument go through verbatim.

\noindent This completes the proof.
\end{proof}

\begin{note}\label{note}
Note that in Theorem~\ref{task3}, we can rename \( y_1 \cdots y_m z_i \) as \( z_i \) to conclude that the following system of equations is partition regular:
\[
\begin{aligned}
x_1 + \cdots + x_n &= z_1 + P_1(z), \\
&\vdots \\
x_1 + \cdots + x_n &= z_r + P_r(z).
\end{aligned}
\]
In addition, for every subset \( F \subseteq \{1,2,\ldots,n\} \), the sum \( \sum_{t \in F} x_t \) also lies in the same cell of the partition.
\end{note}

\begin{thm}
The following result is a polynomial extension of the Sum-Equals-Product theorem in a different direction. We adapt the technique of the proof of Theorem~\ref{task3} to prove it.

Let \( n \in \mathbb{N} \), and let \( \{P_i : i = 1, 2, \ldots, n\} \subset \Po \) be a collection of polynomials. Then the system of equations 
\[
\begin{aligned}
x \cdot y &= z_1 + P_1(z), \\
x \cdot y^2 &= z_2 + P_2(z), \\
&\vdots \\
x \cdot y^n &= z_n + P_n(z)
\end{aligned}
\]
is partition regular. In addition, the set \( \{ x y^i : 1 \leq i \leq n \} \) also lies entirely within a single cell of the partition.
\end{thm}

\begin{proof}
Let \( p \) be a centrally rich ultrafilter and \( A \in p \). Then \( A^\star \in p \). As \( A^\star \) is multiplicatively central, from \cite[Corollary 2.7]{tran}, there exist \( a, d \in \mathbb{N} \) such that 
\[
F = \{ a, d, ad, ad^2, \ldots, ad^n \} \subset A^\star.
\] 
Define 
\[
B = A^\star \cap \bigcap_{s \in F} s^{-1} A^\star \in p.
\] 
Now \( B \) is additively central (since \( p \) is centrally rich), so there exist \( b, c \in \mathbb{N} \) such that 
\[
\left\{ b, c, b - \frac{1}{s} P_i(c) : s \in F,\ i = 1, 2, \ldots, n \right\} \subset B.
\]

Now define:
\begin{enumerate}
    \item \( y = d \);
    \item \( z = c \);
    \item \( x = ab \).
\end{enumerate}

Then for every \( i = 1, 2, \ldots, n \), we define:
\[
\begin{aligned}
z_i &= ad^i \left( b - \frac{1}{ad^i} P_i(c) \right) \\
&= x y^i - P_i(z).
\end{aligned}
\]

Thus, each equation
\[
x y^i = z_i + P_i(z)
\]
is satisfied with all variables lying in the same color class. This completes the proof.
\end{proof}

The following corollary of the above theorem shows that rational functions involving polynomials are also partition regular.

\begin{cor}

Let \( n \in \mathbb{N} \), and \( P, Q \in \mathbb{P} \). Then the equation 
\[
\frac{x + P(d)}{y + Q(d)} = z^n
\]
is partition regular.

In fact, for any \( P, Q, R, S \in \mathbb{P} \), the equation
\[
\frac{w + P(d)}{x + Q(d)} = \left( \frac{y + R(d)}{z + S(d)} \right)^n
\]
is partition regular.
\end{cor}

\subsection{Nonlinear Rado Systems are Partition Regular}

Now we prove that nonlinear Rado systems, as defined earlier, are indeed partition regular.

\begin{thm}\label{ho}
    Nonlinear Rado systems are partition regular.
\end{thm}

\begin{proof}
Let \( A = (a_{i,j}) \) be an \( m \times n \) matrix satisfying the conditions of Definition~\ref{newdef}.  
Let \( p \) be a centrally rich ultrafilter. Since \( p \in K(\beta \mathbb{N}, \cdot) \), every set in \( p \) contains a solution \( \vec{X} \in \mathbb{N}^n \) to the homogeneous linear system \( A\vec{X} = \vec{0} \).

Fix \( B \in p \), and define
\[
B^\star := \{ n \in \mathbb{N} : n^{-1}B \in p \} \in p.
\]
Then \( B^\star \in p \), and there exists \( \vec{X} = \begin{bmatrix} x_1 \\ x_2 \\ \vdots \\ x_n \end{bmatrix} \in (B^\star)^n \) such that \( A\vec{X} = \vec{0} \).

Now define
\[
C := B^\star \cap \bigcap_{i=1}^n x_i^{-1} B^\star \in p.
\]

Let \( P_1, \ldots, P_m \in \mathbb{P} \) be given polynomials. Since \( C \) is additively central, there exist \( a, d \in \mathbb{N} \) such that
\[
\left\{ a, d,\ a - \frac{1}{a_{i,n} x_n} P_i(d) : 1 \leq i \leq m \right\} \subseteq C.
\]

Now for each \( i = 1, 2, \ldots, m \), compute:
\[
\begin{aligned}
& a_{i,1} a x_1 + \cdots + a_{i,n-1} a x_{n-1} + a_{i,n} x_n \left( a - \frac{1}{a_{i,n} x_n} P_i(d) \right) \\
&= a_{i,1} a x_1 + \cdots + a_{i,n-1} a x_{n-1} + a_{i,n} a x_n - P_i(d) \\
&= a \left( a_{i,1} x_1 + \cdots + a_{i,n} x_n \right) - P_i(d) \\
&= 0 - P_i(d) = -P_i(d).
\end{aligned}
\]

Thus, each row of the nonlinear Rado system is satisfied, and all relevant elements lie in the same color class determined by \( p \). Hence, the system is partition regular.
\end{proof}

\begin{example}
As a direct consequence of the previous theorem, the following nonlinear system of equations is partition regular:
\[
\begin{aligned}
&x_1 + 2x_2 - 3y_1 + (z^2 + z) = 0, \\
&2x_1 - x_2 - y_2 + z^3 = 0.
\end{aligned}
\]
\end{example}

\section{Proof of the Necessary Condition}

In this section, we prove Theorem~\ref{N2}. Our proof adapts the classical argument of Rado, as found in~\cite{Graham, Moreira}. The main idea involves the use of the $p$-adic representation of natural numbers. Before presenting the full proof, we require the following preparatory lemma.

\begin{lem}\textup{(\cite[Chapter 3, Lemma 6, Page: 74]{Graham})}\label{N1}
  Let \( j, k \in \mathbb{N} \), and let \( c_1, \dots, c_j \in \mathbb{Z}^k \) be vectors such that \( c_1 \) is not a rational linear combination of the (possibly empty) set \( \{c_2, \dots, c_j\} \). Then there exists a finite set of primes \( F \) such that for any prime \( p \notin F \) and any non-negative integer \( n \), the vector \( p^n c_1 \) is not a linear combination of \( c_2, \dots, c_j \mod p^{n+1} \).

\end{lem}

Now we start our proof.

\begin{proof}[Proof of Theorem \ref{N2}]
Let \( P = \{P_1, \ldots, P_m\} \subset \mathbb{P} \) be a finite set of polynomials, and let \( A = (a_{ij}) \in \mathbb{Z}^{m \times l} \) be an integer matrix. Suppose that for every finite coloring of \( \mathbb{N} \), there exists a monochromatic set \( \vec{X} = \{x_1, \ldots, x_l\} \cup \{z\} \) such that the system
\[
A \vec{X} + P(z) = \vec{0}
\]
holds.

Let \( c_1, \dots, c_{l+n} \) denote the columns of the augmented matrix \( A_{\mathrm{aug}}(P) \). Consider any two disjoint subsets \( I, J \subset \{c_1, \dots, c_{l+n}\} \) with \( I \) nonempty. Then, either the sum \( \sum_{c \in I} c \) lies in the \( \mathbb{Q} \)-linear span of \( J \), or Lemma~\ref{N1} applies and yields a finite set of primes \( F_{I,J} \) satisfying its conclusion. Define \( F \) to be the (finite) union of all such sets \( F_{I,J} \), taken over all choices of disjoint subsets \( I, J \) with \( I \neq \emptyset \). Let \( p \) be any prime not belonging to \( F \).

Let \( \pi : \mathbb{Z} \to \mathbb{Z}/(p\mathbb{Z}) \) denote the natural projection. Define a coloring \( \chi_p : \mathbb{N} \to \{1, \dots, p-1\} \) recursively as follows:
\[
\chi_p(x) = \begin{cases}
\pi(x) & \text{if } \pi(x) \neq 0, \\
\pi(x/p) & \text{if } \pi(x) = 0.
\end{cases}
\]

Let \( (x_1, \dots, x_{l+1}) \in \mathbb{N}^{l+1} \) be a monochromatic vector such that, defining \( \vec{X} = (x_1, \dots, x_l) \) and \( z = x_{l+1} \), we obtain a solution to the equation \( A\vec{X} + P(z) = \vec{0} \). Then there exists some \( a \in \{1, \dots, p-1\} \) such that for each \( i \), the congruence \( x_i \equiv a p^{n_i} \mod p^{n_i+1} \) holds for some \( n_i \in \{0, 1, \dots\} \).

Let \( s+1 \) be the number of distinct values in the set \( \{n_i : i = 1, \dots, l+1\} \). After reordering the columns of \( A \), if needed, we can choose indices \( 0 = d_0 < d_1 < d_2 < \cdots < d_s < d_{s+1} = l+1\) such that
\[
n_i = \begin{cases}
b_0 & \text{if } d_0 < i \leq d_1, \\
b_1 & \text{if } d_1 < i \leq d_2, \\
\vdots & \vdots \\
b_s & \text{if } d_s < i \leq d_{s+1},
\end{cases}
\]
for some strictly increasing sequence \( b_0 < b_1 < \cdots < b_s \).

For each \( 1 \leq i \leq s+1 \), define the index set \( I_i = [d_{i-1}+1, d_i] \). Then
\[
0 = \sum_{i \in I_1} x_i c_i = a p^{b_0} \sum_{i \in I_1} c_i \pmod{p^{b_0+1}}.
\]

When $c_{l+1}\neq \vec 0,$ since higher-degree terms (such as \( x_{l+1}^2, x_{l+1}^3, \dots \)) do not contribute modulo \( p^{b_0+1} \), the computation of the sum over \( I_1 \) is straightforward.
Choosing the prime $p$ enough large, we have condition $(1 ).$

For the case \( c_{l+1} = \vec{0} \), the calculations proceed analogously. The only distinction is that the nonlinear part may contribute terms of the form \( a^i \) for \( i \neq 1 \), while the linear part contributes the coefficient \( a \). Here we need to consider \(\pmod{p^{ib_j+1}} \) for some $0\leq j\leq s.$ Using Lemma \ref{N1}, this gives $(1').$ For detailed clarification see the Note \ref{note}.


For the general case, first define \( c_{t_0} = \vec{0} \). Now for every \( j > 1 \), there exists \( l+2 = t_1 \leq t_k \leq l + n \) such that
\begin{align*}
    0 &= \sum_{i=1}^{l+1} x_i c_i + \sum_{i=2}^{n} c_{l+i} x_{l+1}^{i} \\
      &\equiv \sum_{i=1}^{d_j} x_i c_i + a p^{b_j} \sum_{i=d_j + 1}^{d_{j+1}} c_i + \sum_{i=0}^{k} c_{t_i} x_{l+1}^{i+1} \pmod{p^{b_j+1}}.
\end{align*}

Hence,
\[
    (-a) p^{b_j} \sum_{i \in I_{j+1}} c_i \equiv \sum_{i \in I_1 \cup \cdots \cup I_j} x_i c_i + \sum_{i=0}^{k} c_{t_i} x_{l+1}^{i+1} \pmod{p^{b_j+1}}.
\]

Let \( \tilde{a} \in \mathbb{Z} \) be an integer such that \( a \tilde{a} \equiv 1 \pmod{p^{b_j+1}} \). (Such an integer exists because \( a \) is coprime to \( p \), and hence also to \( p^{b_j+1} \).) Then we obtain the following congruence:
\begin{equation}\label{eqn1}
    p^{b_j} \sum_{i \in I_{j+1}} c_i \equiv -\tilde{a} \sum_{i \in I_1 \cup \cdots \cup I_j} x_i c_i - \tilde{a} \sum_{i=0}^{k} c_{t_i} x_{l+1}^{i+1} \pmod{p^{b_j+1}}.
\end{equation}

Again, suppose \( x_{l+1} \in I_{j+1} \) for some \( 0 \leq j \leq s \). Then \( x_{l+1} = a p^{b_j} \). Let \( 2 \leq u \leq n \). Then there exists \( 0 \leq k \leq s \) such that
\begin{align*}
    0 &= \sum_{i=1}^{l+1} x_i c_i + \left( c_{l+2} x_{l+1}^2 + \cdots + c_{l+n} x_{l+1}^n \right) \\
      &= \sum_{i=1}^{d_{k+1}} x_i c_i + c_{l+2} x_{l+1}^2 + \cdots + c_{l+u} x_{l+1}^u \pmod{p^{u b_j + 1}} \\
      &\equiv \sum_{i=1}^{d_{k+1}} x_i c_i + a^2 c_{l+2} p^{2 b_j} + \cdots + a^u c_{l+u} p^{u b_j} \pmod{p^{u b_j + 1}}.
\end{align*}

Again, since \( (a, p) = 1 \), we let \( \tilde{a} \) denote the inverse of \( a^u \) modulo \( p^{u b_j + 1} \). Then we obtain the congruence
\begin{equation}\label{eqn2}
    p^{u b_j} c_{l+u} \equiv -\tilde{a} \left( \sum_{i=1}^{d_{k+1}} x_i c_i + a^2 c_{l+2} p^{2 b_j} + \cdots + a^{u-1} c_{l+u-1} p^{(u-1) b_j} \right) \pmod{p^{u b_j + 1}}.
\end{equation}

Since \( p \notin F \), it follows from the definition of \( F \) that the previous congruences—equations \eqref{eqn1} and \eqref{eqn2}—imply the following:
\[
c_{d_j+1} + c_{d_j+2} + \cdots + c_{d_{j+1}} \in \operatorname{span}_{\mathbb{Q}} \left( \{ c_i : i \leq d_j \} \cup \{ c_{t_i} : 0 \leq i \leq k \} \right),
\]
and for each \( 2 \leq u \leq n \),
\[
c_{l+u} \in \operatorname{span}_{\mathbb{Q}} \left( \{ c_i : i \leq d_k \} \cup \{ c_{l+1}, c_{l+2}, \dots, c_{l+u-1} \} \right).
\]

\end{proof}

\begin{note}\label{note}
\textbf{Clarification of the proof of (1'):}  
To clarify the argument, let us consider a special case that may arise in the proof. Suppose \( x_{l+1} \equiv a p^{b_0} \pmod{p^{b_0+1}} \) is the only nonzero component modulo \( p^{b_0+1} \), and assume that \( c_{l+1} = \vec{0} \), yet \( c_{l+1} \neq \vec{0} \). Then there must exist some variables \( x_i \) with \( b_1 = 2b_0 \); that is, \( x_i \equiv a p^{2b_0} \pmod{p^{2b_0+1}} \).

On the other hand, from
\[
x_{l+1} \equiv a p^{b_0} \pmod{p^{b_0+1}} \quad \Rightarrow \quad x_{l+1} = a p^{b_0} + r p^{b_0+1}
\]
for some integer \( r \), it follows that
\[
x_{l+1}^2 \equiv a^2 p^{2b_0} \pmod{p^{2b_0+1}}.
\]

Therefore, the equation becomes
\[
\sum_{i \in I_2} c_i a p^{2b_0} + a^2 p^{2b_0} \equiv 0 \pmod{p^{2b_0+1}}.
\]
This suggests that the sum of the columns need not vanish exactly, but instead corresponds to a rational linear combination. Hence, Lemma~\ref{N2} implies condition (1') in this setting.

This reasoning also carries over to the next result.
\end{note}

\subsection{Necessary Condition for Systems of Nonlinear Inhomogeneous Equations}

Let \( m, l \in \mathbb{N} \), and let \( \langle F_i(x_1, \ldots, x_l) \rangle_{i=1}^m \) be a system of inhomogeneous equations — i.e., there is no fixed degree. Define the set of monomial indices
\[
B = \left\{ (i_1, \ldots, i_l) \in \mathbb{N}^l : \text{ the monomial } x_1^{i_1} \cdots x_l^{i_l} \text{ appears in some } F_i \right\}.
\]
Then \( B \) is finite. Define a matrix \( A\left( \langle F_i(x_1, \ldots, x_l) \rangle_{i=1}^m \right) \) where each column corresponds to a monomial from \( B \), and the \( i \)-th row records the coefficient of that monomial in \( F_i \).

Let \( C \) denote the set of columns of this matrix. Suppose \( |C| = l \), and denote the columns by \( c_1, c_2, \ldots, c_l \).

Then, an adaptation of the proof of Theorem~\ref{N2} yields the following necessary condition:

\begin{thm}\label{just1}
Let \( m, l \in \mathbb{N} \), and let \( \langle F_i(x_1, \ldots, x_l) \rangle_{i=1}^m \) be a system of inhomogeneous equations. Suppose that for every finite coloring of \( \mathbb{N} \), the system
\begin{align*}
    F_1(x_1, \ldots, x_l) &= 0, \\
    &\vdots \\
    F_m(x_1, \ldots, x_l) &= 0
\end{align*}
admits a monochromatic nonconstant solution \( (x_1, \ldots, x_l) \in \mathbb{N}^l \). Then there exists a nonempty set \( I_1 \subseteq \{1, 2, \ldots, l\} \), $s\in \N$ such that:
\begin{enumerate}
\item $\{1,2,\ldots ,l\}=\bigcup_{j=1}^sI_j;$
    \item $\vec 0 \in \operatorname{span}_{\mathbb{Q}} \left( \{ c_i : i \in I_1 \} \right);$ and
    \item for every \(2\leq j\leq s\), and for each $k\in I_j,$ 
    \[
    c_k \in \operatorname{span}_{\mathbb{Q}} \left( \left\lbrace c_i : i \in \bigcup_{l=1}^{j-1}I_l\cup (I_j\setminus \{k\}) \right\rbrace \right).
    \]

\end{enumerate}

Note that if the system is homogeneous of a fixed degree, then condition $(1)$ reduces to $\vec 0=\sum_{i\in I_1}c_i.$
\end{thm}

\section{Further Possibilities}

Our current treatment is restricted to systems with finitely many equations and variables. However, infinite linear systems have been studied extensively (see \cite{new1}). A natural extension would be to ask:

\begin{quote}
    \textit{Is it possible to identify infinite nonlinear Rado systems that remain partition regular?}
\end{quote}

   In \cite{HL06, HM}, the authors studied monochromatic solutions to systems of linear equations of the form \( A\vec{X} = \vec{b} \). They showed that such a system is partition regular—that is, it admits a monochromatic solution under every finite coloring of \( \mathbb{N} \)—if and only if it has a constant solution; that is, there exists \( d \in \mathbb{Q} \) such that  
\[
A \begin{pmatrix}
d \\
d \\
\vdots \\
d
\end{pmatrix} = \vec{b}.
\]  
A natural question arises: 

\begin{quote}
{\it Can one formulate a nonlinear analogue of this result?}
\end{quote}

\section*{Acknowledgement}  I sincerely thank Dibyendu De for his valuable suggestions on the previous draft of this paper. I also gratefully acknowledge the support of the NBHM Postdoctoral Fellowship (Reference No: 0204/27/(27)/2023/R \& D-II/11927).

\end{document}